\documentclass{article}
\usepackage{amsmath}
\usepackage{amsfonts}
\usepackage{graphicx, subfigure}
\usepackage[pdftex]{hyperref}
\usepackage[dutch,english]{babel}
\usepackage{amsmath,amssymb}
\usepackage{amsthm}
\usepackage[usenames,dvipsnames]{color}
\usepackage{tikz}
\usetikzlibrary{arrows,automata}
\usepackage{accents}
\usepackage{xcolor}
\usepackage{hyperref}
\usepackage{mathtools}
\usepackage{multicol}
\usepackage{todonotes}

\usepackage{tikz}
\usetikzlibrary{shadows, arrows}

\newcommand{\la}{\lambda}

\DeclareMathOperator{\syn}{Syn}
\DeclareMathOperator{\ifo}{if}
\DeclareMathOperator{\id}{Id}
\DeclareMathOperator{\homc}{End}
\DeclareMathOperator{\nil}{Nil}
\DeclareMathOperator{\im}{Im}
\DeclareMathOperator{\Dim}{dim}

\theoremstyle{plain}
\newtheorem{thr}{Theorem}[section]

\newtheorem{lem}[thr]{Lemma}
\newtheorem{cor}[thr]{Corollary}
\theoremstyle{definition}
\newtheorem{defi}[thr]{Definition}

\theoremstyle{remark}
\newtheorem{remk}{Remark}
\theoremstyle{remark}

\newcommand{\field}[1]{\mathbb{#1}}
\newcommand{\R}{\field{R}}
\newcommand{\N}{\field{N}}
\newcommand{\C}{\field{C}}
\newcommand{\Z}{\field{Z}}

\definecolor{wred}{rgb}{0.7,0.18,0.12}
\definecolor{wgreen}{rgb}{0.1,0.53,0.37}

\numberwithin{equation}{section}

\title{Projection blocks in homogeneous coupled cell networks}
\author{Eddie Nijholt\footnote{\mbox{Department of Mathematics, VU University Amsterdam, The Netherlands, \href{mailto:eddie.nijholt@gmail.com}{eddie.nijholt@gmail.com} }  }, Bob Rink\footnote{ \mbox{Department of Mathematics, VU University Amsterdam, The Netherlands, \href{mailto:b.w.rink@vu.nl}{b.w.rink@vu.nl} }} {} and Jan Sanders\footnote{ \mbox{Department of Mathematics, VU University Amsterdam, The Netherlands, \href{mailto:jan.sanders.a@gmail.com}{jan.sanders.a@gmail.com} }}}
\date{\today}

\begin{document}

\maketitle

\begin{abstract} 
We introduce a special subset of the graph of a homogeneous coupled cell network, called a projection block, and show that the network obtained from identifying this block to a single point can be used to understand the generic bifurcations of the original network. This technique is then used to describe the bifurcations in a generalized feed-forward network, in which the loop can contain more than one cell.
\end{abstract}

\section{Introduction}
In this paper we consider homogeneous coupled cell network vector fields. These are vector fields of the general form

\begin{equation}
\begin{split}
\dot{x}_1 = &f(x_{\sigma_1(1)}, \dots  x_{\sigma_{n}(1)}, \la)\\
\dot{x}_2 = &f(x_{\sigma_1(2)}, \dots  x_{\sigma_{n}(2)},\la)\\
&\vdots \\
\dot{x}_N = &f(x_{\sigma_1(N)}, \dots  x_{\sigma_{n}(N)},\la) \,,
\end{split} 
\end{equation}
where the variables $x_i$ are elements of the same vector space $V$, the $\sigma_i$ are functions from the set $\{1, \dots N\}$ to itself, and $f$ is a smooth map from an open set $V^n \times \Omega \subset V^n \times \R^p $ to $V$. We are interested in generic bifurcations of such vector fields, which means we vary the bifurcation parameter $\la \in \Omega \subset \R^p$, and ask about qualitative changes in for example the number of steady state points or periodic orbits. See also \cite{e1}, \cite{e2}, \cite{e3}, \cite{e4}, \cite{e5}, \cite{e6} and \cite{e7}. A known example of such a system is the so-called feed-forward network, studied in for example \cite{gs}, \cite{feedf}, \cite{feedf2} and \cite{feedf3}. It is given by the equations

\begin{equation}\label{ff1}
\begin{split}
\dot{x}_1 = &f(x_1, x_2, \dots  x_{n-1}, x_n, \la)\\
\dot{x}_2 = &f(x_2, x_3, \dots  x_{n}, x_n, \la)\\
&\vdots \\
\dot{x}_n = &f(x_{n}, x_n, \dots  x_{n}, x_n, \la) \,,
\end{split} 
\end{equation}
and can be depicted by the network of figure \ref{fig:pic1}.

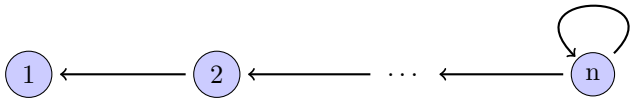
\begin{figure}[h]
\centering
\begin{tikzpicture}
\node[circle, draw, fill=blue!20] at (0,0) (a) {1} ; 
\node[circle, draw, fill=blue!20] at (2.5,0) (b) {2};
\node[] at (5,0) (c) {\dots} ;
\node[circle, draw, fill=blue!20] at (7.5,0) (d) {n}  ;
\draw[<-, thick, shorten >= 0.1 cm, shorten <= 0.1 cm] (a) edge (b) (b) edge (c) (c) edge (d) (d) edge [loop] (d);
\end{tikzpicture}\caption{A feed-forward network with $n$ cells.}\label{fig:pic1}
\end{figure}
More precisely, figure \ref{fig:pic1} shows the dependence of each cell as given by the second entry of $f$. For example, the equation for cell $1$ is given by $\dot{x}_1 = f(x_1, x_2, \dots  x_{n-1}, x_n, \la)$. Because $f$ is evaluated here at $x_2$ in the second entry, we say that the state of cell $1$ depends on the state of cell $2$ (in a way described by the second entry of $f$). Therefore,  the network of figure  \ref{fig:pic1} contains an arrow from cell $2$ to cell $1$. The other entries of $f$ are then obtained by concatenating the black arrows a fixed number of times, and by adding self-loops to describe the dependence of each cell on their own state by the first entry of $f$.\\
\indent In \cite{feedf} it is shown that in the case of $V = \R$ and $\Omega \subset \R$ the system \eqref{ff1} has generically one of two steady state bifurcations from a fully synchronous point, i.e. a point with $x_1  = \dots = x_n$. These bifurcations are a fully synchronous saddle node bifurcation and a synchrony breaking bifurcation. This latter bifurcation has, in addition to a fully synchronous branch, $n-1$ branches scaling as $|\la|^{l_1}$ to $|\la|^{l_{n-1}}$, where we have set $l_i := \frac{1}{2^{i-1}}$. \\
\indent Let us also look at the following network

\begin{equation}\label{ff2}
\begin{split}
\dot{x}_1 = &f(x_1, x_2, x_3, x_4, \la)\\
\dot{x}_2 = &f(x_2, x_3, x_4, x_3, \la)\\
\dot{x}_3 = &f(x_3, x_4, x_3, x_4, \la)\\
\dot{x}_4 = &f(x_4, x_3, x_4, x_3, \la) \quad,
\end{split} 
\end{equation}
depicted by figure \ref{fig:pic2}. Again, we have only shown the dependence of each cell through the second entry of the response function $f$. The third and fourth correspond to concatenating the given arrows two and three times. 

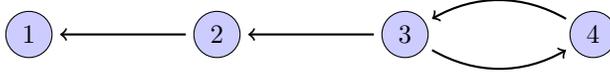
\begin{figure}[h]
\centering
\begin{tikzpicture}
\node[circle, draw, fill=blue!20] at (0,0) (a) {1} ; 
\node[circle, draw, fill=blue!20] at (2.5,0) (b) {2};
\node[circle, draw, fill=blue!20] at (5,0) (c) {3} ;
\node[circle, draw, fill=blue!20] at (7.5,0) (d) {4}  ;
\draw[<-, thick, shorten >= 0.1 cm, shorten <= 0.1 cm] (a) edge (b) (b) edge (c) (c) edge[bend left] (d) (d) edge[bend left]  (c);
\end{tikzpicture}\caption{The network of equation \ref{ff2}.}\label{fig:pic2}
\end{figure}

The network \eqref{ff2} is similar to the feed-forward network \eqref{ff1}. However, in a feed-forward network there is one cell whose state depends only on itself and that influences all the others, whereas in the network of \eqref{ff2} there are two cells that are only influenced by themselves and by each other, but that in their turn feed into the other cells. Nevertheless, by identifying the cells $3$ and $4$ one obtains the $3$-cell feed-forward network back:

\begin{equation}\label{ff3}
\begin{split}
\dot{x}_1 = &g(x_1, x_2, x_3,\la)\\
\dot{x}_2 = &g(x_2, x_3, x_3,  \la)\\
\dot{x}_3 = &g(x_3, x_3, x_3, \la) \quad,
\end{split} 
\end{equation}
for $g(x_1, x_2, x_3,\la) := f(x_1,x_2,x_3,x_3,\la)$. Note that any (smooth) map \\ \noindent $g: V^3 \times \Omega \rightarrow V$  can be obtained by restricting some (smooth) map \\ \noindent $f: V^4 \times \Omega \rightarrow V$ to the space $\{x_3 = x_4\}$. As it turns out, the generic steady state bifurcations of \eqref{ff3} are generic in \eqref{ff2} as well. We will show that this is a consequence of a more general theorem, relating the bifurcations in a network with a set of cells that only influence each other (such as the cells $3$ and $4$ in \eqref{ff2}) to that of the same network with these cells identified (i.e. the system \eqref{ff3}). More specifically, a projection block will be a set of cells in a network that only feel each other and that influence every other cell, in a way to be made precise later on. The main theorem of this paper can be roughly summarized as the following, see theorem \ref{main} and corollary \ref{main2}.

\begin{thr}\label{main0}
Let $B$ be a projection block in a homogeneous coupled cell network $N$ and let $N_P$ denote the network equal to $N$ with the block $B$ identified to a point. We may naturally identify any network vector field for the network $N_P$ as a network vector field for the network $N$ restricted to some invariant subspace. Suppose a smooth network vector field $\gamma_f$ for the network $N_P$ has a fully synchronous steady state point $x_0$. Suppose furthermore that the center subspace of the linearization $D\gamma_f(x_0)$ at $x_0$ contains no non-trivial fully synchronous points. Let us denote a center manifold of this system by $M_c$. 

Then in the set of smooth network vector fields for $N$ that, when restricted to the phase space of $N_P$ are equal to $\gamma_f$, there exists an open and dense set of vector fields with center manifold equal to $M_c$ as well. In particular, any local bifurcation that occurs in the $N_p$-system occurs in all these $N$-systems, without additional branches of bounded solutions.
\end{thr} 
\noindent This result relies mainly on an investigation of the possible center subspaces of a network vector field. It turns out these can be described in terms of the invariant subspaces of a certain monoid-representation, which is the main theme of this paper.\\
\indent The rest of this paper is set up as follows. In section $2$ we introduce some of the techniques we will be using in subsequent sections, most notably that of a fundamental network and that of center manifold reduction for homogeneous coupled cell networks. Likewise, section $3$ serves to introduce the details of monoid-representations needed throughout this paper. In section $4$ we introduce quotient monoids, which we relate to quotient networks in section $5$. In section $6$ we introduce the notion of a projection block and formulate and prove the main result of this paper. In section $7$ we then work out the machinery we have developed on a generalization of a feed-forward network.

\section{Preliminaries}
Recall that a homogeneous coupled cell network is a system of the general form

\begin{equation}\label{nn}
\begin{split}
\dot{x}_1 = &f(x_{\sigma_1(1)}, \dots  x_{\sigma_{n}(1)})\\
\dot{x}_2 = &f(x_{\sigma_1(2)}, \dots  x_{\sigma_{n}(2)})\\
&\vdots \\
\dot{x}_N = &f(x_{\sigma_1(N)}, \dots  x_{\sigma_{n}(N)}) \quad,
\end{split} 
\end{equation}
The underlying network structure $N$ can be described by the set of nodes $C:= \{1, \dots N\}$ and the set of interaction functions $\Sigma = \{\sigma_1, \dots \sigma_n \}$, $\sigma_i: C \rightarrow C$. Whereas $\Sigma$ a priori need not satisfy any additional conditions, this often means that the set of systems of the form \eqref{nn} is too intractable to work with. For example, both the composition and the Lie-bracket of two vector fields of the form \eqref{nn} need not be of this form any longer. We will therefore slightly enlarge the class of vector fields we consider, by enlarging the set $\Sigma$ to include the identity $\id: C \rightarrow C$ and all compositions of two or more functions that appear in $\Sigma$. Note that this new class of network vector fields includes all the original ones, as the response function $f$ may only depend formally on the new variables. This setup is described in more detail in  \cite{new}, \cite{norm}, \cite{repr} and \cite{feedf}, where it is shown that this larger class of vector fields is closed under taking compositions and Lie-brackets.\\
By the above discussion, we will henceforth always assume that $\Sigma$ is closed under multiplication and contains the identity $\id: C \rightarrow C$. This means that $\Sigma$ is a monoid. In particular, we may construct the regular representation $(V^n,A_{\Sigma})$ of $\Sigma$. Here, the action of $\Sigma$ is given by $(A_{\sigma}X)_{\tau} = X_{\tau \circ \sigma}$ for $\sigma, \tau \in \Sigma$ and we identify $V^n$ with $\bigoplus_{\sigma \in \Sigma} V$. It can be seen that indeed $A_{\sigma} \circ A_{\tau} =A_{\sigma \circ \tau}$ for all $\sigma, \tau \in \Sigma$ and that $A_{\id} = \id|_{V^n}$. As it turns out, the equivariant vector fields on $(V^n,A_{\Sigma})$ are exactly the coupled cell network vector fields $\Gamma_f$ given by

\begin{equation}\label{funref}
\begin{split}
\Gamma_f(X)_{\sigma_1} = &f(X_{\sigma_1\circ \sigma_1}, \dots  X_{\sigma_n\circ \sigma_1})\\
\Gamma_f(X)_{\sigma_2} = &f(X_{\sigma_1\circ \sigma_2}, \dots  X_{\sigma_n\circ \sigma_2})\\
&\vdots \\
\Gamma_f(X)_{\sigma_n} = &f(x_{\sigma_1\circ \sigma_n}, \dots  X_{\sigma_n\circ \sigma_n}) \quad.
\end{split} 
\end{equation}
We will explain monoid-representations, including the regular representation in more detail in the next section. The motivation for considering the regular representation is as follows. Given a network vector field \eqref{nn} we call a subspace of the phase space a \textit{synchrony space} when it is given by equations of the form $x_i = x_j$ for certain nodes $i$ and $j$. In other words, a synchrony space is just a poly-diagonal subspace. A synchrony space is then called \textit{robust} if it is a flow-invariant space for any network vector field \eqref{nn}. Since we have defined network vector fields $\Gamma_f$ on $(V^n,A_{\Sigma})$, we may equally well speak of (robust) synchrony spaces on this latter space. It follows that a synchrony space of $(V^n,A_{\Sigma})$ is robust if and only if it is flow-invariant for any equivariant vector field. Perhaps counter-intuitive is that such a robust synchrony space need in general not be respected by the symmetries $A_{\sigma}$ of $(V^n,A_{\Sigma})$. For any cell $p \in C$, the synchrony space \\
$\syn_{N,p}:= \{X_{\sigma_i} = X_{\sigma_j} \text{ if } \sigma_i(p) = \sigma_j(p)\} \subset (V^n,A_{\Sigma})$ is robust, hence flow-invariant for any vector field $\Gamma_f$. Moreover, if the network $N$ has a cell $p \in C$ such that $\{\sigma(p) : \sigma \in \Sigma\} = C$, then the vector fields $\Gamma_f|_{\syn_{N,p}}$ are exactly the coupled cell network vector fields of \eqref{nn}. We therefore say that the network of \eqref{funref} is the \textit{fundamental network} of the network $N$. Similarly, we call $\Gamma_f$ the \textit{fundamental network vector field} of \eqref{nn}. We say that $p$ is a \textit{fully dependent cell} of the network $N$ if the condition  $\{\sigma(p) : \sigma \in \Sigma\} = C$ is satisfied. Unless stated otherwise, all of the networks in this article will have such a cell, so that they can be realized as sub-systems of their fundamental network vector fields. More on homogeneous coupled cell networks can be found in for example \cite{new}, \cite{fibr}, \cite{norm}, \cite{repr}, \cite{feedf}, \cite{gs} and \cite{ld}. \\
In \cite{new}, the authors have developed a center manifold theorem for homogeneous coupled cell networks around a fully synchronous point. The first result is that for a fundamental network vector field there always exists a local center manifold that is invariant under the monoid-symmetries. The second is one that classifies all the vector fields one might obtain by reducing $\Gamma_f$ to its center manifold, i.e. the \textit{reduced vector fields}. To describe these, we say that an $A_{\Sigma}$-invariant subspace $W$ of $(V^n,A_{\Sigma})$ is \textit{complementable} if there exists an $A_{\Sigma}$-invariant subspace $W'$ such that $V^n = W \oplus W'$. Whereas any invariant space under the linear action of a finite group always has an invariant complement, this is in general not true for the representation of a finite monoid. The result of \cite{new} is that the possible reduced vector fields of the fundamental network are exactly those that are conjugate to an equivariant vector field on a complementable invariant space of $(V^n,A_{\Sigma})$. Furthermore, such a conjugacy can always be found in such a way that it preserves the robust synchrony spaces. Therefore, the reduced vector fields of \eqref{nn} are exactly those of $\Gamma_f$ restricted to $\syn_{N,p}$.

%jhgfh

\newpage

\section{Representation theory of monoids}
In this section we will briefly explain the definitions and results about monoid representations relevant to this paper.  Recall that a monoid is a set $\Sigma$ together with an associative multiplication and a unit $e \in \Sigma$. In other words, there exists a map $\circ: \Sigma \times \Sigma \rightarrow \Sigma$ such that $(a \circ b) \circ c = a \circ (b \circ c)$ for all $a,b,c \in \Sigma$ and $a \circ e = e \circ a = a$ for all $a \in \Sigma$.  We will furthermore always assume in this paper that $\Sigma$ is finite, i.e. contains only finitely many elements.

\begin{defi}
1. Let $\Sigma$ be a finite monoid and let $W$ be a finite dimensional real vector space. A \textit{representation of $\Sigma$ on W} is a map $A$ from $\Sigma$ to $\mathcal{L}(W,W)$ (the space of linear maps from $W$ to itself) such that

\begin{itemize}
\item $A( \sigma ) A(\tau )  = A(\sigma \circ \tau)$  for all $\sigma, \tau \in \Sigma$ 
\item $A(e) = \id_W$ .
\end{itemize}
We will often denote $A(\sigma)$ by $A_{\sigma}$ for $\sigma \in \Sigma$ and will use the notation $(W,A_{\Sigma})$ to denote a representation of $\Sigma$ on $W$. \\
\noindent 2. Given a representation $(W,A_{\Sigma})$, a linear subspace $U \subset W$ is called \textit{invariant} if $A_{\sigma}u \in U$ for all $u \in U$ and $\sigma \in \Sigma$. We will call the invariant space $U$ \textit{complementable} if there exists an invariant subspace $U' \subset W$ such that $W = U \oplus U'$. \\
\noindent 3. Given two representations $(W,A_{\Sigma})$ and $(W',A'_{\Sigma})$ of the monoid $\Sigma$, we call a map $f: W \rightarrow W'$ \textit{equivariant} if $f \circ A_{\sigma} = A'_{\sigma} \circ f$ for all $\sigma \in \Sigma$.
\end{defi}

\noindent Given a vector space $V$ and a monoid $\Sigma$, we can construct the representation $(\bigoplus_{\sigma \in \Sigma}V, A_{\sigma})$ as follows: A vector $X$ in $\bigoplus_{\sigma \in \Sigma}V$ can uniquely be written as $X =\sum_{\sigma \in \Sigma} X_{\sigma}$ for $X_{\sigma} \in V$. As such, we define the action of $\Sigma$ on $\bigoplus_{\sigma \in \Sigma}V$ by $(A_{\sigma}X)_{\tau} = X_{\tau \circ \sigma}$ for all $\tau, \sigma \in \Sigma$ and $X \in \bigoplus_{\sigma \in \Sigma}V$. It can easily be verified that this indeed defines a representation. If $\#\Sigma = n$ then we will denote this representation by $(V^n, A_{\Sigma})$ and we will refer to it as the \textit{regular representation} of the monoid $\Sigma$. As we have noted in section $2$ the equivariant maps from $(V^n, A_{\Sigma})$ to itself are exactly the admissible vector fields of the fundamental network of a coupled cell network with monoid $\Sigma$. \\

\noindent If we are given a linear equivariant map $B$ from a representation $(W,A_{\Sigma})$ to itself, then it is not hard to see that $\ker(B)$ and $\im(B)$ are examples of invariant spaces in $(W,A_{\Sigma})$. Likewise for all $\mu \in \R$ and $\lambda \in \C \backslash \R$ the span of the eigenvectors of $B$, $\ker(B - \mu \id)$ and $\ker((B-\la \id)(B-\bar{\la} \id))$ are examples of invariant spaces. In general none of the above examples are complementable though. Other examples of invariant spaces are the generalized eigenspaces of $B$, $W_{B,\la} := \ker((B-\la \id)^n(B-\bar{\la} \id)^n)$ for $\la \in \C$ and with $n := \Dim(W)$. These examples are complementable, as we have the direct sum decomposition

\begin{equation} \label{eigendecom}
W = \bigoplus_{\la \in EV(B)} W_{B,\la} \, ,
\end{equation}
where $EV(B) \subset \C$ denotes the set of eigenvalues if $B$ (containing only one of each complex pair). We also note that any decomposition $W = U \oplus V$ into invariant spaces gives rise to a projection with image $U$ and kernel $V$, which is readily seen to be an equivariant map. Conversely, any equivariant projection gives rise to a decomposition into its image and its kernel, both of which are invariant and therefore complementable. \\

\noindent 
Another important notion for the representation of monoids is that of an \textit{indecomposable} representation. 

\begin{defi}
A representation $(U,A_{\Sigma})$ is called \textit{indecomposable} when it holds that $U \not= \{0\}$ and that $U$ cannot be written in a non-trivial way as the direct sum of two invariant spaces.
\end{defi}
\noindent Note that indecomposable representations can still contain non-trivial invariant spaces, albeit without invariant complements. By iteratively decomposing any representation $(W,A_{\Sigma})$ into invariant spaces, one can see that any representation can be written as the direct sum of indecomposable representations (this process terminates by the fact that $W$ is finite dimensional, and must do so in a decomposition into indecomposable representations). It is shown in \cite{repr} that any such decomposition is furthermore unique. In other words, given two decompositions of $(W,A_{\Sigma})$ into indecomposable representations

\begin{equation}
W = \bigoplus_{i=1}^p W_i = \bigoplus_{j=1}^q W'_j \,,
\end{equation}
it follows that $p = q$ and that $W_i$ is isomorphic to $W'_i$ for all $i \in \{1, \dots p\}$, possibly after reordering (we say that two representations are isomorphic when there exists an invertible linear equivariant map between them, from which it follows that the inverse is a linear equivariant map as well).\\
\indent Given an equivariant linear map $B$ from an indecomposable representation $(U,A_{\Sigma})$ to itself, it follows from the decomposition \eqref{eigendecom} that $B$ necessarily has either one real eigenvalue or one pair of complex conjugate eigenvalues. Consequently, $B$ is either nilpotent or invertible. This result is often called Schur's lemma, after an analogous (but stronger) result for the representation of finite groups over fields of characteristic $0$. It can then be seen that in the case of an indecomposable representation the sum of two equivariant nilpotent maps, as well as the composition of an equivariant nilpotent map with any equivariant map is again nilpotent. In other words, if we denote by $\homc_{\Sigma}(U)$ the linear equivariant maps from  $(U,A_{\Sigma})$ to itself, and by $\nil_{\Sigma}(U) \subset \homc_{\Sigma}(U)$ the ones that are nilpotent, then $\nil_{\Sigma}(U)$ is a two-sided ideal in $\homc_{\Sigma}(U)$. Consequently, the quotient ring $\homc_{\Sigma}(U)/\nil_{\Sigma}(U)$ is a finite dimensional  division algebra over $\R$. This has important consequences, as a theorem by Frobenius states that the only finite dimensional  division algebras over $\R$ are isomorphic to either $\R$, $\C$ or $\mathbb{H}$ (the quaternions). As a result, we get the following classification of indecomposable representations. 

\begin{defi}
The indecomposable representation $(U,A_{\Sigma})$ is called
\begin{itemize}
\item of \textit{real type} if $\homc_{\Sigma}(U)/\nil_{\Sigma}(U) \cong \R$,
\item of \textit{complex type} if $\homc_{\Sigma}(U)/\nil_{\Sigma}(U) \cong \C$ or
\item of \textit{quaternionic type} if $\homc_{\Sigma}(U)/\nil_{\Sigma}(U) \cong \mathbb{H}$.
\end{itemize}
\end{defi}
As a last remark we give the following theorem. Its proof can be found in for example \cite{repr}.

\begin{thr}
Let $B$ be an equivariant linear map from the indecomposable representation $(U,A_{\Sigma})$ to the indecomposable representation $(U',A'_{\Sigma})$, and let $B'$ be an equivariant linear map from $(U',A'_{\Sigma})$ to  $(U,A_{\Sigma})$. If $B \circ B'$ is invertible, then both $B$ and $B'$ are isomorphisms of representations. In particular, if $(U,A_{\Sigma})$ is not isomorphic to $(U',A'_{\Sigma})$ then $B \circ B'$ is nilpotent.
\end{thr}

%jhsdjsdhgdsjhes

\section{Quotient monoids and the regular representation}
In this section we formulate and prove a result that relates the regular representation of a monoid to that of a so-called quotient monoid. This result will be an important ingredient for the main theorem \ref{main0} as formulated in the introduction, where it is used to relate the generic bifurcations of a network to that of a specific quotient network.

\begin{defi}[Homomorphisms of monoids and quotient monoids]
Let $\Sigma$ and $T$ be monoids. A function $\pi: \Sigma \rightarrow T$ is called a \textit{homomorphism of monoids} when it satisfies

\begin{itemize}
\item $\pi(e_{\Sigma}) = e_T$ for the units $e_{\Sigma} \in \Sigma$ and $e_T \in T$,
\item $\pi(\sigma \cdot \sigma') = \pi(\sigma) \cdot \pi(\sigma')$ for all $\sigma, \sigma' \in \Sigma$, where multiplication is to be understood in $\Sigma$ respectively $T$.
\end{itemize}
We say that $T$ is a \textit{quotient monoid} of $\Sigma$ if there exists a surjective homomorphism of monoids $\pi: \Sigma \rightarrow T$. 
\end{defi}

\begin{remk}\label{unit}
Given a surjective function $\pi$ between monoids $\Sigma$ and $T$ such that $\pi(\sigma \cdot \sigma') = \pi(\sigma) \cdot \pi(\sigma')$ for all $\sigma, \sigma' \in \Sigma$, it follows immediately that $\pi(e_{\Sigma}) = e_T$. Namely, we have that $\pi(e_{\Sigma})\cdot\pi(\sigma) = \pi(e_{\Sigma}\cdot \sigma) = \pi(\sigma)$ and likewise that $\pi(\sigma)\cdot \pi(e_{\Sigma}) = \pi(\sigma)$ for all $\sigma \in \Sigma$. By surjectivity of $\pi$ we have that $\{\pi(\sigma) : \sigma \in \Sigma\} = T$ and hence we see that $\pi(e_{\Sigma})$ is a unit element in $T$. Note furthermore that an element that is a right unit or a left unit (and specifically both) in a monoid is necessarily the unit of this monoid, as its right respectively left product with the unit would otherwise be ill-defined. \hfill $\triangle$
\end{remk}

\begin{thr}\label{inclu}
Let $\pi: \Sigma \rightarrow T$ be a surjective homomorphism of monoids, so that $T$ is a quotient monoid of the finite monoid  $\Sigma$. Let $(V^m, A_{T})$ and $(V^n, A_{\Sigma})$ be their regular representations, respectively. The synchrony space $\syn_{\pi} := \{X_{\sigma} = X_{\sigma'} \ifo \pi(\sigma) = \pi(\sigma') \} \subset (V^n, A_{\Sigma})$ is robust. Furthermore, it is an invariant space of $(V^n, A_{\Sigma})$ on which the action of $\Sigma$ depends only on the images $\pi(\sigma)$ for  $\sigma \in \Sigma$. In particular, this invariant synchrony space can be seen as a representation space of $T$ and as such it is in fact isomorphic to the regular representation $(V^m, A_{T})$ of $T$.
\end{thr}

\begin{proof}
We will start by showing that $\syn_{\pi}$ is a robust synchrony space. In \cite{norm} it is shown that this is the case when the partition dictating which nodes are identified in the synchrony space is respected by the elements of $\Sigma$. See also \cite{balan1}. In our case the partition corresponding to the synchrony space is $\{\pi^{-1}(\tau): \tau \in T \}$. Hence, for every $\sigma \in \Sigma$ we need to show that $\sigma \cdot \sigma'$ and $\sigma \cdot \sigma''$ are in the same set $\pi^{-1}(\tau)$ if $\sigma'$ and $\sigma''$ are in the same set $\pi^{-1}(\tau')$. In other words, we need to show that $\pi(\sigma \cdot \sigma') = \pi(\sigma \cdot \sigma'')$ if $\pi(\sigma') = \pi(\sigma'')$. However, this is immediate as $\pi(\sigma \cdot \sigma')  = \pi(\sigma)\cdot \pi(\sigma')= \pi(\sigma)\cdot \pi(\sigma'') =  \pi(\sigma \cdot \sigma'')$. \\
Next, we show that $\syn_{\pi}$ is an invariant space. Recall that $A_{\sigma}$ is given by $(A_{\sigma}X)_{\sigma'} = X_{\sigma' \cdot \sigma}$ for $\sigma, \sigma' \in \Sigma$. Hence, we see that\\ $\syn_{\pi} = \{X_{\sigma} = X_{\sigma'} \ifo \pi(\sigma) = \pi(\sigma') \} $ is invariant if for every $\sigma \in \Sigma$ we have that $\pi(\sigma' \cdot \sigma) = \pi(\sigma'' \cdot \sigma) $  whenever $\pi(\sigma') = \pi(\sigma'')$. This condition is indeed satisfied because $\pi(\sigma' \cdot \sigma)= \pi(\sigma')\cdot \pi( \sigma) = \pi(\sigma'')\cdot \pi(  \sigma)  = \pi(\sigma'' \cdot \sigma) $. \\
Finally, a similar argument shows that $A_{\sigma}$ and $A_{\sigma'}$ act the same on $\syn_{\pi}$ if $\pi(\sigma) = \pi(\sigma')$. Therefore, we may write $A_{\pi^{-1}(\tau)} := A_{\sigma}|_{\syn_{\pi}}$ for any $\sigma \in \pi^{-1}(\tau)$. Let us furthermore set  $\{X_{\pi^{-1}(\tau)} \}_{\tau \in T}$ as  coordinates for $\syn_{\pi}$,  where we have that $X_{\sigma} = X_{\pi^{-1}(\tau)}$ whenever $\pi(\sigma) = \tau$. Note that we use here that none of the sets $\pi^{-1}(\tau)$ for $\tau \in T$ is empty, as $\pi$ is assumed to be surjective. We then see that the action of $\Sigma$ on $\syn_{\pi}$ can be written as $(A_{\pi^{-1}(\tau)}X)_{\pi^{-1}(\tau')}   = X_{\pi^{-1}(\tau'\cdot \tau)}$. Hence,  identifying $X_{\pi^{-1}(\tau)}$ with $X_{\tau}$ and $A_{\pi^{-1}(\tau)}$ with $A_{\tau}$ for every $\tau \in T$, we see that $\syn_{\pi} \subset (V^n, A_{\Sigma})$ can be identified with $(V^m, A_{T})$ as representations of $T$.
\end{proof}
\noindent In light of the previous theorem, we may identify $(V^m, A_{T})$ with $\syn_{\pi} \subset (V^n, A_{\Sigma})$. Using this identification, we can associate to any linear subspace $W$ of $(V^n, A_{\Sigma})$ a linear subspace $W \cap \syn_{\pi}$ of $(V^m, A_{T})$.  The following theorem tells us that the function $W \mapsto W \cap \syn_{\pi}$ respects the structure of a decomposition into indecomposable representations.

\begin{thr}\label{submon}
Let $V^n = \bigoplus_{i = 1}^k W_i$ be a decomposition of $(V^n, A_{\Sigma})$ into indecomposable representations. Then 

\begin{equation}
V^m = \bigoplus_{W_i} W_i \cap \syn_{\pi}
\end{equation}
is a decomposition of $(V^m, A_{T})$ into indecomposable representations. Conversely, if  $V^m = \bigoplus_{i = 1}^l U_i$ is a decomposition of  $(V^m, A_{T})$ into indecomposable representations, then there exists a decomposition $V^n = \bigoplus_{i = 1}^k W_i$ of $(V^n, A_{\Sigma})$ into indecomposable representations (with $k \geq l$) such that  $U_i = W_i \cap \syn_{\pi}$ for $1 \leq i \leq l$.  Furthermore, we have the following relations between $W_i$ and $W_i \cap \syn_{\pi}$

\begin{itemize}
\item If  $W_i \cap \syn_{\pi} \not= \{0\}$ then $W_i$ and $W_i \cap \syn_{\pi}$ are of the same type (real, complex or quaternionic).
\item if  $W_i \cap \syn_{\pi} \not= \{0\}$ and $W_j \cap \syn_{\pi} \not= \{0\}$ then $W_i$ is isomorphic to $W_j$ if and only if $W_i \cap \syn_{\pi}$ is isomorphic to  $W_j \cap \syn_{\pi}$.
\item if   $W_i \cap \syn_{\pi} \not= \{0\}$ but  $W_j \cap \syn_{\pi} = \{0\}$ then $W_i$ and $W_j$ are not isomorphic.
\end{itemize}
\end{thr}
\noindent The main ingredient of the proof will be the following lemma.

\begin{lem}\label{ext}
For any equivariant vector field $\Gamma_f^{T}$ on $(V^m, A_T) \simeq \syn_{\pi}$ we can find an equivariant vector field  $\Gamma_{\tilde{f}}^{\Sigma}$ on $(V^n, A_{\Sigma})$ such that $\Gamma_{\tilde{f}}^{\Sigma}|_{\syn_{\pi}} = \Gamma_f^{T}$.
\end{lem}

\begin{proof}
Let us assume we are given the equivariant vector field $\Gamma_f^T$ on $(V^m, A_T)$ corresponding to the function $f = (\Gamma_f^T)_{e_T}: V^m \rightarrow V$, where $e_T$ denotes the unit in $T$. For every $\tau \in T$ we may pick one element $\sigma_{\tau} \in \Sigma$ such that $\pi(\sigma_{\tau}) = \tau$. In other words, we pick one representative out of every class  $\pi^{-1}(\tau)$. Note that none of the sets $\pi^{-1}(\tau)$ is empty, as $\pi$ is surjective. Writing $\Sigma = \{ \sigma_1 \dots \sigma_n\}$ and $T = \{\tau_1, \dots \tau_m\}$, we then define the function $\tilde{f}: V^n \rightarrow V$ given by \\ $\tilde{f}(X_{\sigma_1} \dots X_{\sigma_n}) := f(X_{\sigma_{\tau_1}} \dots  X_{\sigma_{\tau_m}})$. In particular, it follows that $\tilde{f}|_{\syn_{\pi}} = f$. We claim that $\Gamma_{\tilde{f}}^{\Sigma}$ is an equivariant vector field satisfying $\Gamma_{\tilde{f}}^{\Sigma}|_{\syn_{\pi}} = \Gamma_f^{T}$. First of all, because $\syn_{\pi}$ is a robust synchrony space of $(V^m, A_T)$ we see that $\Gamma_{\tilde{f}}^{\Sigma}$ indeed sends elements of $\syn_{\pi}$ to itself. Next, because $\syn_{\pi}$ is an invariant space on which the action of $\Sigma$ coincides with that of $T$, we may conclude that $\Gamma_{\tilde{f}}^{\Sigma}|_{\syn_{\pi}} $ is a $T$-equivariant vector field on $(V^m, A_T)$. In particular, we have that $\Gamma_{\tilde{f}}^{\Sigma}|_{\syn_{\pi}} = \Gamma_g^T$ for some function $g:V^m \rightarrow V$. Finally, it remains to show that $g = f$. Per definition, we have that $g = (\Gamma_{g}^T)_{e_T}$. Using the identification between $\syn_{\pi}$ and $(V^m, A_T)$, we see that $g = (\Gamma_{\tilde{f}}^{\Sigma}|_{\syn_{\pi}})_{\sigma}$ for any element $\sigma \in \pi^{-1}(e_T)$. In particular, we know that the unit in $\Sigma$ is contained in $\pi^{-1}(e_T)$. We therefore see that $g = (\Gamma_{\tilde{f}}^{\Sigma}|_{\syn_{\pi}})_{e_\Sigma} = \tilde{f}|_{\syn_{\pi}} = f$. This proves the lemma.
\end{proof}
\noindent Given an equivariant vector field on $(V^m, A_T)$, we will generally use a tilde to denote an equivariant extension on $(V^n, A_{\Sigma})$ in the spirit of lemma \ref{ext}. Note that it follows from the proof of lemma \ref{ext} that the extension of a linear vector field can be taken to be linear as well. Furthermore, by fixing the choice of representatives in the proof of lemma \ref{ext}, we see that any smooth family of vector fields on $(V^m, A_T)$ can be extended to a smooth family of vector fields on $(V^n, A_{\Sigma})$.

\begin{proof}[Proof of theorem \ref{submon}]
Let $V^n = \bigoplus_{i = 1}^k W_i$ be a decomposition of $(V^n, A_{\Sigma})$ into invariant spaces. The projections $P_i :V^n \rightarrow W_i \subset V^n$ on the different components are equivariant maps and therefore leave the space $\syn_{\pi}$ invariant. In particular, for any element $v \in \syn_{\pi}$ we see that $v = \sum_{i=1}^k P_i(v)$ is a decomposition into elements of $W_i \cap \syn_{\pi}$. Since any decomposition into elements of $W_i$, hence into elements of $W_i \cap \syn_{\pi}$ is unique, it follows that

\begin{equation}\label{decom}
\syn_{\pi} = \bigoplus_{W_i \cap \syn_{\pi}\not= \{0\}} W_i \cap \syn_{\pi}
\end{equation}
is a decomposition of $\syn_{\pi} \simeq (V^m, A_T)$ into invariant spaces. \\
Let us now assume one of the components $W_i$ is indecomposable. We will show that this implies that $W_i \cap \syn_{\pi}$ is indecomposable as well, by assuming the converse and arriving at a contradiction. Suppose we can write $W_i  \cap \syn_{\pi} = U_0 \oplus U_1$, where $U_0$ and $U_1$ are invariant spaces both unequal to $\{0\}$. Denote by $P^T_1: \syn_{\pi} \rightarrow U_1 \subset \syn_{\pi}$ the projection onto $U_1$ corresponding to the decomposition $W_i  \cap \syn_{\pi} = U_0 \oplus U_1$ and to decomposition $\eqref{decom}$. In particular, we see that $P^T_1|_{U_0} = 0$ and that $P^T_1|_{U_1} = \id|_{U_1}$. Because $P^T_1$ is an equivariant map, it follows from lemma \ref{ext} that there exists an equivariant map $\widetilde{P^T_1}$  on $(V^n, A_{\Sigma})$ such that $\widetilde{P^T_1}|_{\syn_{\pi}} = P^T_1$. Let us continue to denote by  $P_i :V^n \rightarrow W_i \subset V^n$ the projection onto $W_i$ corresponding to the decomposition $V^n = \bigoplus_{i = 1}^k W_i$. The map $H:= P_i \circ \widetilde{P^T_1}|_{W_i}:W_i \rightarrow W_i$ is an equivariant map from the indecomposable representation space $W_i$ to itself.  Therefore, it is either invertible or nilpotent. This is a contradiction though, as $P_i\circ \widetilde{P^T_1}|_{U_1} = P_i \circ \id|_{U_1} = \id_{U_1}$, so $H$ is not nilpotent, and $P_i \circ \widetilde{P^T_1}|_{U_0} = 0$, so $H$ is not invertible. We conclude that indeed $W_i \cap \syn_{\pi}$ has to be indecomposable if the component $W_i$ is. In particular, if $V^n = \bigoplus_{i = 1}^k W_i$ is a decomposition into indecomposable sub-representations then so is expression \eqref{decom}. \\
Now suppose that conversely we are given a decomposition $\syn_{\pi} = \bigoplus_{i = 1}^l U_i$ into (positive dimensional) sub-representations. Let us choose a set $\{\la_i\}$ of $l$ distinct values in $\R$  and define the linear map $P_{\{\la_i\}}: \syn_{\pi} \rightarrow \syn_{\pi}$ given by $P_{\{\la_i\}}|_{U_i} = \la_i \cdot \id|_{U_i}$ for all $i \in \{1, \dots l\}$. Because the spaces $U_i$ are invariant it follows that the map $P_{\{\la_i\}}$ is equivariant. In particular, we may conclude from lemma \ref{ext} that there exists an equivariant map $\widetilde{P}_{\{\la_i\}}$  from $(V^n, A_{\Sigma})$ to itself that restricts to $P_{\{\la_i\}}$. Let us denote by  $EV(\widetilde{P}_{\{\la_i\}})$ the set of eigenvalues of the map $\widetilde{P}_{\{\la_i\}}$ (containing only one of each complex pair). Note that $\{ \la_i\}$ is included in $EV(\widetilde{P}_{\{\la_i\}})$, as $\{ \la_i\}$ is the set of eigenvalues of $P_{\{\la_i\}} = \widetilde{P}_{\{\la_i\}}|_{\syn_{\pi}}$. Denoting by $W_{\mu}$ the generalized eigenspace of $\widetilde{P}_{\{\la_i\}}$ corresponding to the eigenvalue  $\mu \in EV(\widetilde{P}_{\{\la_i\}})$, we get a decomposition of $(V^n, A_{\Sigma})$ into invariant spaces

\begin{equation}\label{decom2}
V^n= \bigoplus_{\mu \in EV(\widetilde{P}_{\{\la_i\}}) } W_{\mu} \quad.
\end{equation}
Note that we have $W_{\la_i} \cap \syn_{\pi} = U_i$ for all $i \in \{1, \dots l\}$ and $W_{\mu} \cap \syn_{\pi}= \{0\}$ if $\mu \notin \{\la_i\}$. Hence, the decomposition \eqref{decom2} gives rise to the decomposition $\syn_{\pi} = \bigoplus_{i = 1}^l U_i$ in the sense of the first part of the theorem. Moreover, we may further decompose $W_{\la_i} = \oplus_{j=1}^p W_{\la_i}^j$ into indecomposable representations, which gives rise to a decomposition of  $U_i$.  Assuming $U_i$ is indecomposable, we conclude that there is a $j \in \{1, \dots p \}$ such that $U_i = W_{\la_i}^j \cap \syn_{\pi}$ and such that $U_i \cap (W_{\la_i}^q \cap \syn_{\pi}) = W_{\la_i}^q \cap \syn_{\pi} = \{0\}$ for all $q \not= j$. Hence, the components $W_i$ may be chosen to be indecomposable themselves.\\
Next, we show that $W_i \cap \syn_{\pi}$ and $W_i$ are of the same type if both are indecomposable and if $W_i \cap \syn_{\pi} \not= \{0\}$. Given any equivariant map $\phi: W_i \rightarrow W_i$, we may extend this map to an equivariant map $\phi': (V^n, A_{\Sigma}) \rightarrow (V^n, A_{\Sigma})$ by setting $\phi'|_{W_i} = \phi$ and $\phi'|_{W_j} = 0$ for $j\not=i$. Because the map $\phi'$ is an equivariant vector field on $(V^n, A_{\Sigma})$, it sends the robust synchrony space $\syn_{\pi}$ to itself. From this we conclude that $\phi$ restricts to an equivariant map from $W_i \cap \syn_{\pi}$ to itself. In other words, if we denote by $\homc_{\Sigma}(W_i)$ and $\homc_{\Sigma}(W_i \cap \syn_{\pi})$ the space of $\Sigma$-endomorphisms on $W_i$ respectively $W_i \cap \syn_{\pi}$, then restriction defines a linear map

\begin{equation} \label{mapR}
\begin{split}
R: \homc_{\Sigma}(W_i) &\rightarrow \homc_{\Sigma}(W_i \cap \syn_{\pi})\\
\phi &\mapsto \phi|_{W_i \cap \syn_{\pi}} \quad.
\end{split}
\end{equation}
Moreover, if $\phi \in \homc_{\Sigma}(W_i)$ is nilpotent then so is $\phi|_{W_i \cap \syn_{\pi}}$. Hence, $R$ factors through to a map

\begin{equation} \label{map[R]}
\begin{split}
[R]: \homc_{\Sigma}(W_i)/ \nil_{\Sigma}(W_i)&\rightarrow \homc_{\Sigma}(W_i \cap \syn_{\pi})/ \nil_{\Sigma}(W_i \cap \syn_{\pi})\\
[\phi] &\mapsto [R(\phi)] \quad,
\end{split}
\end{equation}
where $\nil_{\Sigma}(W_i) \subset \homc_{\Sigma}(W_i)$ and $\nil_{\Sigma}(W_i \cap \syn_{\pi}) \subset \homc_{\Sigma}(W_i \cap  \syn_{\pi})$ denote the nilpotent elements. We will now show that $[R]$ is a bijection, thereby proving that $W_i \cap \syn_{\pi}$ and $W_i$ are of the same type. Injectivity of $[R]$ follows from the fact that $R(\phi) := \phi|_{W_i \cap  \syn_{\pi} }$ is invertible whenever $\phi \in \homc_{\Sigma}(W_i)$ is. As for surjectivity, this is true for $[R]$ if it is true for $R$. Therefore, let\\ $\psi \in \homc_{\Sigma}(W_i \cap \syn_{\pi})$ be given, we will construct an element $\phi \in \homc_{\Sigma}(W_i)$ such that $R(\phi) = \psi$. For this purpose, we first construct an equivariant map $\psi' :\syn_{\pi} \rightarrow \syn_{\pi}$ such that $\psi'|_{W_i \cap  \syn_{\pi} } = \psi$, for example by letting $\psi'$ vanish on an invariant complement of $W_i \cap  \syn_{\pi}$ in $ \syn_{\pi}$. By lemma \ref{ext} there exist an equivariant extension $\widetilde{\psi'}$ of $\psi'$ to $(V^n, A_{\Sigma})$. The map $\phi := P_i \circ \widetilde{\psi'}|_{W_i}$ is then an element of $\homc_{\Sigma}(W_i)$. Furthermore, we have $\phi|_{W_i \cap  \syn_{\pi}} = P_i \circ \widetilde{\psi'}|_{W_i \cap  \syn_{\pi}} = P_i \circ \psi'|_{W_i \cap  \syn_{\pi}} = P_i \circ \psi = \psi$.  This proves that $R$ and therefore $[R]$ is surjective and hence that $W_i$ and $W_i \cap \syn_{\pi}$ are of the same type. \\
Next, suppose $W_i$ and $W_j$ are isomorphic indecomposable representations. We will show that there exists an invertible equivariant map from $W_i \cap \syn_{\pi}$ to $W_j \cap \syn_{\pi}$. To this end, let $\alpha$ be an isomorphism from $W_i$ to $W_j$. As before, we can expand $\alpha$ to an equivariant map $\alpha'$ on $(V^n, A_{\Sigma})$ by letting $\alpha'$ vanish on some complement of $W_i$. The map $\alpha'$ then sends the space $\syn_{\pi}$ to itself. In particular, we see that $\alpha'$ sends $W_i \cap \syn_{\pi}$ to $W_j \cap \syn_{\pi}$. Moreover, since $\alpha'|_{W_i} = \alpha$, we see that $\alpha'|_{W_i \cap \syn_{\pi}}$ is injective. Repeating this procedure with $\alpha$ replaced by $\alpha^{-1}$ and with the roles of $W_i$ and $W_j$ reversed, we see that there exist injective equivariant maps from $W_i \cap \syn_{\pi}$ to $W_j \cap \syn_{\pi}$ and vice versa. Hence, both are bijections.  This shows that if $W_i$ and $W_j$ are isomorphic and if $W_i \cap \syn_{\pi} \not= \{0\}$, then $W_j\cap \syn_{\pi} \not= \{0\}$ and $W_i \cap \syn_{\pi}$ and $W_j \cap \syn_{\pi}$ are isomorphic.  \\
Finally, we will show that if $W_i$ and $W_j$ are indecomposable, and if $W_i \cap \syn_{\pi} \not= \{0\}$ and $W_i \cap \syn_{\pi} \not= \{0\}$ are isomorphic, then $W_i$ and $W_j$ are isomorphic as well. For this purpose, let $\beta_{ij}: W_{i} \cap \syn_{\pi} \rightarrow W_j \cap \syn_{\pi} $ and\\ $\beta_{ji}: W_j \cap \syn_{\pi} \rightarrow W_i \cap \syn_{\pi}$ be isomorphisms. As before, we can extend $\beta_{ij}$ and $\beta_{ji}$ to equivariant maps $\beta'_{ij}, \beta'_{ji}: \syn_{\pi} \rightarrow \syn_{\pi}$ by letting them vanish on some compliment of $W_i \cap \syn_{\pi}$ respectively $W_j \cap \syn_{\pi}$. Next, by lemma \ref{ext} there exist maps $\widetilde{\beta'_{ij}}$ and $\widetilde{\beta'_{ji}}$ on $(V^n, A_{\Sigma})$ that restrict to $\beta'_{ij}$ and $\beta'_{ji}$ on $\syn_{\pi}$ and hence to  $\beta_{ij}$ and $\beta_{ji}$ on $W_i \cap \syn_{\pi}$ and $W_j \cap \syn_{\pi}$, respectively. Therefore, the maps $B_{ij}: = P_j \circ \widetilde{\beta'_{ij}}|_{W_i}: W_i \rightarrow W_j$ and $B_{ji}: = P_i \circ \widetilde{\beta'_{ji}}|_{W_j}: W_j \rightarrow W_i$ are equivariant and likewise restrict to  $\beta_{ij}$ and $\beta_{ji}$ on $W_i \cap \syn_{\pi}$ and $W_j \cap \syn_{\pi}$. We finish the proof by noting that the map $B_{ji} \circ B_{ij}: W_i \rightarrow W_j \rightarrow  W_i$ is an equivariant map that is the composition of two equivariant maps between indecomposable representations. Hence, it is either nilpotent or we have that $W_i$ and $W_j$ are isomorphic. The former can however not be, as $B_{ji} \circ B_{ij}$ restricts to the invertible function $\beta_{ji}\circ \beta_{ij}: W_i \cap \syn_{\pi} \rightarrow W_i \cap \syn_{\pi}$. We conclude that indeed $W_i$ and $W_j$ are isomorphic. This proves the theorem. 

\end{proof}
\noindent By the strong correlation between indecomposable representations and generic bifurcations as laid out in the previous sections, theorem \ref{submon} can be read as a result relating the generic bifurcations of two homogeneous networks. More specifically, such a result holds if the monoid of the one network is a quotient monoid of the other network. In the next section we will further explore this relation.

\section{Quotient networks and quotient monoids}
Let $N$ be a homogeneous coupled cell network with nodes $C$ and monoid $\Sigma$. Recall that a balanced partition of $N$ is a partition $P=\{P_1 \dots P_s\}$ of the set $C$ such that the elements of $\Sigma$ respect $P$. In other words, for every $\sigma_i \in \Sigma$ it holds that if $q,r \in C$ are two nodes from the same partition class $P_j$, then $\sigma_i(q)$ and $\sigma_i(r)$ are two elements from some same partition class $P_k$. We will often use $[q]$ to denote the partition class containing a node $q$, i.e. we have $[q] = P_j$ if and only if $q \in P_j$. In this notation the partition being balanced means that $[q] = [r]$ implies $[\sigma_i(q)] = [\sigma_i(r)]$ for every $\sigma_i \in \Sigma$. Hence, it follows that $\Sigma$ naturally factors through to a set of functions from the set of partition classes to itself, by setting $\sigma_i([q]) := [\sigma_i(q)]$. We will furthermore identify two functions $\sigma_i$ and $\sigma_j$ if they act the same on the set of partition classes. I.e. we write $\sigma_i \sim \sigma_j$ if and only if $[\sigma_i(q)] = [\sigma_j(q)]$ for all $q \in C$. The corresponding equivalence class of functions $[\sigma_i] = [\sigma_j]$ can then be seen as one well-defined function from the set of partition classes to itself by writing $[\sigma_i]([q]) := [\sigma_i(q)]$ for every $q \in C$. To summarize, we may define a new homogeneous coupled cell network $N_P$, whose set of nodes is $C_P:=\{[q]: q\in C\} = \{P_j\}_{j=1}^s$ and whose arrows are described by the functions $\Sigma_P := \{[\sigma_i]: \sigma_i \in \Sigma\}$. Note that $[\id]$ is the identity function on $C_P$, where $\id \in \Sigma$ is the identity on $C$. \\
It can be seen that $\Sigma_P$ is closed under composition and is therefore a monoid itself. Namely, we have 
\begin{equation}\label{calc}
[\sigma_i]\circ [\sigma_j]([q])= [\sigma_i]([\sigma_j(q)]) =[(\sigma_i\circ \sigma_j)(q)] = [\sigma_i\circ \sigma_j]([q]) \quad,
\end{equation}
for all $q \in C$ and $\sigma_i, \sigma_j \in \Sigma$. In fact, it follows from equation \eqref{calc} that the map 

\begin{equation}\label{pip}
\begin{split}
\pi_P: \Sigma &\rightarrow \Sigma_P\\
\sigma_i &\mapsto [\sigma_i]
\end{split}
\end{equation}
satisfies $\pi_P(\sigma_i \circ \sigma_j) = \pi_P(\sigma_i) \circ \pi_P(\sigma_j)$ for all $\sigma_i, \sigma_j \in \Sigma$. Combined with the fact that $\pi_P$ is a surjection, we get the following result.

\begin{thr}
Let $N = (C,\Sigma)$ be a homogeneous coupled cell network and let $N_P = (C_P, \Sigma_P)$ be a quotient network of $N$ corresponding to a balanced partition $P$ of $N$. Then, $\Sigma_P$ is a quotient monoid of $\Sigma$ via the surjection $\pi_P$.
\end{thr}

\begin{remk}\label{synspac}
Suppose that $p \in C$ is a fully dependent cell for the network $N$. In other words, we have that $\{\sigma_i(p): \sigma_i \in \Sigma\} = C$. It follows that $\{[\sigma_i]([p]): [\sigma_i] \in \Sigma_P\} = \{[\sigma_i(p)]: \sigma_i \in \Sigma\} = C_P$. Hence, we see that $[p]$ is a fully dependent cell for the network $N_P$. \\
For this reason, we may identify the network $N_P$ with a synchrony space $\syn_{N_P,[p]}$ of its regular representation $(V^m, A_{\Sigma_P})$, where we have set $m:=\# \Sigma_P$. Specifically, this synchrony space is given by 
\begin{equation}
\syn_{N_P,[p]} := \{X_{[\sigma_i]} = X_{[\sigma_j]} \text{ if } [\sigma_i]([p]) = [\sigma_j]([p])\} \subset (V^m, A_{\Sigma_P}) \quad.
\end{equation}
Furthermore, by the previous section we may identify $(V^m, A_{\Sigma_P})$ with a synchrony space $\syn_{\pi_P} \subset (V^n, A_{\Sigma})$, given by
\begin{equation}
\syn_{\pi_P} := \{X_{\sigma_i} = X_{\sigma_j} \text{ if } [\sigma_i] = [\sigma_j]\} \subset (V^n, A_{\Sigma}) \quad.
\end{equation}
Therefore we can realize the network $N_P$ as the synchrony space $\syn_{N_P,[p]}\cap \syn_{\pi_P}$ of $(V^n, A_{\Sigma})$. This latter synchrony space is explicitly given by 
\begin{equation} \label{NPpi}
\begin{split}
\syn_{N_P,[p]}\cap \syn_{\pi_P}\:= &\,\{X_{\sigma_i} = X_{\sigma_j} \text{ if } [\sigma_i]([p]) = [\sigma_j]([p])\} = \\
 &\,\{X_{\sigma_i} = X_{\sigma_j} \text{ if } [\sigma_i(p) ]= [\sigma_j(p)]\} \subset (V^n , A_{\Sigma}) \quad.
\end{split}
\end{equation}
There is however a second way of identifying $N_p$ as a synchrony space of $(V^n, A_{\Sigma})$. Namely by first identifying it with a synchrony space $\syn_P$ of the network $N$, and by then identifying $N$ with the synchrony space
\begin{equation}
\syn_{N,p} :=  \{X_{\sigma_i} = X_{\sigma_j} \text{ if } \sigma_i(p) = \sigma_j(p)\} \subset (V^n, A_{\Sigma}) \quad.
\end{equation}
By this procedure $N_p$  corresponds to the synchrony space
\begin{equation}\label{PN}
\syn_{P} \cap \syn_{N,p} :=  \{X_{\sigma_i} = X_{\sigma_j} \text{ if } [\sigma_i(p)] = [\sigma_j(p)]\} \subset (V^n, A_{\Sigma}) \quad.
\end{equation}
In particular we see that the expressions \eqref{NPpi} and \eqref{PN} agree, meaning that the two identifications of the network $N_P$ in $(V^n, A_{\Sigma})$ coincide. 
\hfill $\triangle$
\end{remk}

\begin{remk}\label{inpor}
Suppose we are given a homogeneous coupled cell network $N$ (with a fully dependent cell) and a quotient network $N_P$. To understand the bifurcations in a fully synchronous point of the network $N_P$ we may use center manifold reduction in the space $(V^m,A_{\Sigma_P})$ corresponding to the fundamental network of $N_p$. In particular, the possible reduced vector fields on $(V^m,A_{\Sigma_P})$ are exactly all equivariant vector fields on a complementable subspace $W \subset V^m$. The possible reduced vector fields for $N_p$ are then exactly these vector fields restricted to $W \cap \syn_{N_P,[p]}$. \\
By theorems \ref{inclu} and \ref{submon} we may identify the space $(V^m,A_{\Sigma_P})$ with a robust synchrony space $\syn_{\pi_P} \subset (V^n,A_{\Sigma})$ and find a complementable subspace $W' \subset V^n$ such that $W = W' \cap \syn_{\pi_P}$. Furthermore, if $W$ decomposes into indecomposable representations as 

\begin{equation}
W = \bigoplus_{i\in I} {W}_i^{n_i} \quad,
\end{equation}
for some finite counting set $I$, then $W'$ can be chosen to decompose into indecomposable representations as

\begin{equation}
W' = \bigoplus_{i\in I} {W'}_i^{n_i} \quad,
\end{equation}
where we have that $W_i$ and $W'_i$ are of the same type (i.e. real, complex or quaternionic) for all $i \in I$. It follows from lemma \ref{ext} that the reduced vector fields of $N_P$ are the equivariant vector fields on some sub-representation $W' \subset V^n$ restricted to  $W' \cap \syn_{\pi_P} \cap \syn_{N_P,[p]}$. \\
Likewise, reduced vector fields for the network $N$ are equivariant vector fields on $W'$ restricted to $W' \cap \syn_{N,p}$. Since we know from remark \ref{synspac} that the robust synchrony spaces $\syn_{\pi_P} \cap \syn_{N_P,[p]}$ and $\syn_P \cap \syn_{N,p}$ coincide, we have that $(W' \cap \syn_{N,p}) \cap \syn_{P} = W' \cap \syn_{N,p} \cap \syn_P = W' \cap \syn_{\pi_P} \cap \syn_{N_P,[p]}$. We conclude from this that the reduced vector fields of $N_P$ are exactly those of $N$ restricted to the synchrony space $\syn_P$. In particular, the possible dynamics on the center manifold of an $N_P$ system can be obtained by restricting the possible dynamics on the center manifold of an $N$ system to the synchrony space $\syn_P$.\\
If $(V^n, A_{\Sigma})$ furthermore decomposes into distinct indecomposable representations, then it is known that a one-parameter steady state bifurcation generically occurs along one indecomposable representation of real type. See \cite{repr}. Now, from theorem \ref{submon} it follows that $(V^m, A_{\Sigma_P})$ decomposes into distinct indecomposable representations whenever $(V^n, A_{\Sigma})$ does. Moreover, it follows that $W' \cap \syn_{\pi_P}$ is of real type whenever $W'$ is. From this we conclude that in the case of distinct indecomposable representations, the generic one-parameter steady state bifurcations of $N_P$ are exactly those of $N$ restricted to $\syn_P$. It is believed by the authors that the condition of distinct indecomposable representations can be dropped. Furthermore, it is believed that in the event of more bifurcation parameters, there are similar results about the generalized kernel and center subspace being generically a number of indecomposable representations of specific types. This would further translate the generic bifurcations on $N$ to the generic ones on $N_P$.
\hfill $\triangle$
\end{remk}
%\newpage
\section{Reduction by projection blocks}

In view of remark \ref{inpor}, it makes sense to look for complementable subspaces $W'$ of $(V^n,A_{\Sigma})$ such that $W' \cap \syn_{N,p} = W' \cap \syn_{N,p} \cap \syn_P$. In that case, the reduced vector fields corresponding to $W'$ on the network $N$ are exactly those corresponding to the sub-representation $W' \cap \syn_{\pi_P}$ on the network $N_P$. In particular, the bifurcations corresponding to $W'$ on the network $N$ are then exactly those on $N_P$ corresponding to $W' \cap \syn_{\pi_P}$. In this section we will describe a class of networks that admit a quotient network on which certain of the sub-representations indeed coincide.

\begin{defi}[Blocks and projection blocks] \label{blok}
Let $N$ be a homogeneous coupled cell network with nodes $C$ and monoid $\Sigma$. A subset of nodes $B \subset C$ is called a \textit{block} if there are no arrows in the graph of $N$ going from a source outside of $B$ to a target inside of $B$. In other words, $B$ is a block if and only if $\sigma_i(b) \in B$ for all $b \in B$ and $\sigma_i \in \Sigma$.\\
A block $B$ is called a \textit{projection block} if there furthermore exists an element $\kappa \in \Sigma$ such that $\kappa(C) = B$ and $\kappa(B) = B$.
\end{defi}
\noindent A block $B$ in a homogeneous coupled cell network $N$ naturally gives rise to a balanced partition. Namely, we say that $q \sim r$ if and only if $q,r \in B$. It is balanced because $q,r \in B$ implies $\sigma_i(q),\sigma_i(r) \in B$ for all $\sigma_i \in \Sigma$, per definition of a block. The resulting quotient network of $N$ corresponding to this balanced partition can be obtained from $N$ by identifying the points in $B$ with a single point $[B]$. Note that we then have $[\sigma_i]([B]) = [B]$ for all $\sigma_i \in \Sigma$. \\
A projection block roughly means that there are some colors of arrows that restrict to a bijection on the block, and whose concatenations connect every point in the network to this block. More precisely, every point in the network can be traced to this block by following arrows of these colors in reverse direction. The monoid element $\kappa$ from definition \ref{blok} is then found as the product of sufficiently many terms corresponding to arrows of these colors. This is the content of the following theorem.

\begin{thr}\label{pict}
Let $B$ be a block in a  homogeneous coupled cell network $N$ with nodes $C$ and monoid $\Sigma$. Suppose that $\Pi \subset \Sigma$ is a generating set for $\Sigma$. Then, $B$ is a projection block if and only if there exists a subset $\Theta = \{\theta_i\}_{i=1}^t \subset \Pi$ satisfying

\begin{itemize}
\item $\theta_i(B) = B$ for all $\theta_i \in \Theta$.
\item For every point $q \in C$ there exists a finite sequence $\{\theta_{i_1} \dots \theta_{i_s}\}$ of elements in $\Theta$ such that $(\theta_{i_1} \circ \dots \circ \theta_{i_s})(q) \in B.$
\end{itemize}
\end{thr}

\begin{proof}
We fix the generating set $\Pi$ of $\Sigma$. First, we assume there exists a subset  $\Theta = \{\theta_i\}_{i=1}^t \subset \Pi$ such that the conditions of theorem \ref{pict} hold. We want to construct an element $\kappa \in \Sigma$ satisfying $\kappa(C) = B$ and $\kappa(B) = B$ as in definition \ref{blok}, and we will do so inductively. First, we note that if $\sigma \in \Sigma$ and $\tau \in \Sigma$  satisfy $\sigma(B) = B$ and $\tau(B) = B$, then we also have that $(\sigma \circ \tau)(B) = B$. Now, if $B$ exactly equals $C$ then $B$ is always a projection block, by setting $\kappa = \id$. Hence, we next assume that $B \not= C$ and choose a point $q_0 \in C \setminus B$. By assumption, there exists a sequence $\theta_I := \theta_{I_1} \dots \theta_{I_s}$ of elements in $\Theta$ such that $\theta_I(q_0) \in B$. It follows that $\theta_I(B) = B$, and so if $\theta_I(C) = B$ we are done by setting $\kappa = \theta_I$. Note that

 \begin{equation}\label{sets1}
 \theta_I(C) \setminus B = \theta_I(C \setminus (B \cup \{q_0\})) \setminus B \quad,
 \end{equation}
 by the fact that $\theta_I(B \cup \{q_0\}) = B$. From equation \eqref{sets1} it follows that 
 
 \begin{equation}
 \#(\theta_I(C) \setminus B ) \leq \# (C \setminus (B \cup \{q_0\})) < \#(C \setminus B ) \quad.
 \end{equation}
 Next, we choose an element  $q_1 \in \theta(C) \setminus B$ and a sequence $\theta_{\tilde{I}}$ such that $\theta_{\tilde{I}}(q_1) \in B$. It follows that 
 
 \begin{equation}\label{sets2}
(\theta_{\tilde{I}} \circ \theta_I)(C) \setminus B = \theta_{\tilde{I}}[\theta_I(C) \setminus (B \cup \{q_1\})] \setminus B \quad,
 \end{equation}
from which we again see that 

 \begin{equation}
 \#((\theta_{\tilde{I}} \circ \theta_I)(C) \setminus B ) \leq \# (\theta_I(C) \setminus (B \cup \{q_1\})) < \#(\theta_I(C) \setminus B )\quad.
 \end{equation}
For convenience, we will redefine $\theta_I$ to be $\theta_{\tilde{I}} \circ \theta_I$. Repeating this procedure, we get a sequence of sets $\theta_I(C) \setminus B$ strictly decreasing in size. Because $C$ only has finitely many elements, we eventually get $\theta_I(C) \setminus B = \emptyset$. Hence, we see that $\theta_I(C) \subset B$. Because we also have $\theta_I(B) = B$, it follows that $\theta_I(C) = B$. Therefore, setting $\kappa:= \theta_I$ we see that $B$ is indeed a projection block.\\
Conversely, if $B$ is a projection block, we may write $\kappa = \sigma_I := \sigma_{i_1} \circ \dots \circ \sigma_{i_s}$ for elements $\sigma_{j}$ in the generating set $\Pi$. It follows that $\sigma_I(q) \in B$ for every node $q \in C$. Hence, we may define $\Theta \subset \Pi$ to be the set of all $\sigma_j$ appearing in $\sigma_I$. It remains to show that $\sigma_j(B) = B$ for every $\sigma_j \in \Theta$. However, we are given that $\kappa|_B = \sigma_{I}|_B$ is a bijection from $B$ to itself $B$. Moreover, as any element of $\Sigma$ maps $B$ into itself, we may write  $\sigma_I|_B = (\sigma_{i_1} \circ \dots \circ \sigma_{i_s})|_B =  \sigma_{i_1}|_B \circ \dots \circ \sigma_{i_s}|_B$. From this it follows that all of the $\sigma_j|_B$ are bijections from $B$ to itself. This proves the theorem.
\end{proof}
\noindent Theorem \ref{pict} tells us that, in order to determine whether or not a block is a projection block, one only has to look at any set of generators for $\Sigma$. In particular, only at those elements of this set of generators that restrict to a bijection on $B$. The block is then a projection block if and only if this subset of generators connects every node to the block.\\
The following theorem gives the motivation for considering projection blocks in homogeneous coupled cell networks. We recall the setting. If $B$ is a projection block in a homogeneous coupled cell network $N$, then we will denote by $P$ the balanced partition corresponding to $B$ and by $N_P$ the corresponding reduced network. As usual,  let $(V^n, A_{\Sigma})$ denote the fundamental network of $N$ and $(V^m, A_{\Sigma_P})$ denote the fundamental network of $N_P$. As we have seen this latter space can be identified as the invariant synchrony space $\syn_{\pi_P}$ of the former. Furthermore, using a fully dependent cell $p \in C$ we have seen that we may retrieve the network vector fields of $N$ and $N_P$ respectively by restricting to the subspaces $\syn_{N,p} \subset (V^n, A_{\Sigma})$ and $\syn_{N_P,[p]}\subset (V^m, A_{\Sigma_P}) \cong \syn_{\pi_P}$.

\begin{thr}\label{main}
Let $B$ be a projection block in a homogeneous coupled cell network $N$. There exists a decomposition 

\begin{equation}
V^n = W \oplus W'
\end{equation}
into invariant spaces such that

\begin{equation}
W \cap \syn_{N,p} =  W \cap \syn_{N_P,[p]}
\end{equation}
and 
\begin{equation}
W' \cap \syn_{\pi_P} =  \syn_0 := \{X_{\sigma} = X_{\tau} \, \forall \sigma, \tau \in \Sigma_P\} \subset (V^m, A_{\Sigma_P})\, .
\end{equation}
\end{thr}
\noindent To prove theorem \ref{main} we first need two lemmas. The first one gives a better motivation for the name 'projection block'.

\begin{lem}\label{idemp}
Let $B$ be a block in a homogeneous coupled cell network $N$ with monoid $\Sigma$. $B$ is a projection block if and only if there exists an element $\iota \in \Sigma$ such that $\iota$ is idempotent, i.e $\iota \circ \iota =\iota$, and such that $\iota(C) = B$.
\end{lem}

\begin{proof}
First we assume $B$ is a projection block. Let $\kappa \in \Sigma$ be an element satisfying $\kappa(C) = \kappa(B) =  B$. It follows that $\kappa^l(C) =  \kappa^l(B) =  B$ for all $l \in \N_{>0}$, where we have set $\kappa^l := \kappa \circ \dots \circ \kappa$ ($l$ times). Next, because $\Sigma$ is finite, it follows that there exist constants $M,N \in \N_{>0}$ such that $\kappa^M = \kappa^{M+N}$. From this we see that $\kappa^{M'} = \kappa^{M'+N}$ for all $M' \geq M$. In particular, choosing $s \in \N_{>0}$ such that $sN \geq M$ we see that $\kappa^{sN} = \kappa^{(s+1)N} = \dots \kappa^{2sN}$. Hence, setting $\iota:=\kappa^{sN}$ we see that indeed $\iota \circ \iota =\iota$ and $\iota(C) = B$. \\
Conversely, $\iota$ satisfies $\iota(C) = B$. Hence, for any element $b \in B$ there exists an element $c \in C$ such that $\iota(c) = b$. It follows that $\iota(b) = \iota^2(c) = \iota(c) = b$. From this we conclude that $\iota(B) = B$. Setting $\iota = \kappa$ then proves that $B$ is a projection block, which proves the theorem.
\end{proof}
\noindent The next lemma states that an idempotent element in a monoid gives rise to a splitting of the regular representation into two invariant spaces.

\begin{lem}\label{idemproj}
Let $\iota$ be an idempotent element of a monoid $\Sigma$ and let $(V^n, A_{\Sigma})$ be the regular representation of $\Sigma$. The map \\
$B_{\iota}: (V^n, A_{\Sigma}) \rightarrow (V^n, A_{\Sigma})$ defined by $(B_{\iota}X)_{\sigma} = X_{\iota \circ \sigma}$ for $\sigma \in \Sigma$ is an equivariant projection.
\end{lem}

\begin{proof}
First we show that the map $B_{\iota}$ is a projection. Because $\iota$ is idempotent it follows that

\begin{equation}
\begin{split}
&((B_{\iota})^2X)_{\sigma} = (B_{\iota}(B_{\iota}(X)))_{\sigma} = (B_{\iota}(X))_{\iota \circ \sigma} =\\ 
&X_{\iota \circ \iota \circ \sigma} = X_{\iota \circ \sigma} =  (B_{\iota}X)_{\sigma} \quad,
\end{split}
\end{equation} 
for all $\sigma \in \Sigma$ and $X \in V^n$. \\
Next we show equivariance. For all $\sigma, \tau \in \Sigma$ and $X \in V^n$ we have 

\begin{equation}
\begin{split}
(A_{\sigma}(B_{\iota}(X)))_{\tau} = (B_{\iota}(X))_{\tau \circ \sigma} = X_{\iota \circ \tau \circ \sigma} \quad, \\
(B_{\iota}(A_{\sigma}(X)))_{\tau} = (A_{\sigma}(X))_{\iota \circ \tau} = X_{\iota \circ \tau \circ \sigma} \quad.
\end{split}
\end{equation} 
From this we see that $A_{\sigma}B_{\iota} = B_{\iota} A_{\sigma}$ for all $\sigma \in \Sigma$, and hence that $B_{\iota}$ is an equivariant map. This proves the claims of the lemma.
\end{proof}

\begin{remk}\label{splitt}
Recall that an equivariant projection $P$ on a representation space $U$ gives rise to a decomposition  
$U = \im(P) \oplus \ker(P)$ into invariant subspaces. In particular, for the map $B_{\iota}$ we get a decomposition into the invariant spaces

\begin{equation}
\im(B_{\iota}) = \{X_{\sigma} = X_{\tau} \text{ if } \iota \circ \sigma = \iota \circ \tau \} \quad,
\end{equation}
and

\begin{equation}
\ker(B_{\iota}) = \{X_{\iota \circ \sigma} = 0 \text{ } \forall \sigma \in \Sigma\} \quad.
\end{equation}
These will be important in the proof of theorem \ref{main}.
\hfill $\triangle$
\end{remk}

%jshdgjshdgf

\begin{proof}[Proof of theorem \ref{main}]
By lemma \ref{idemp}, there exists an idempotent element $\iota \in \Sigma$ such that $\iota(C) = B$. It follows then from lemma \ref{idemproj} and remark \ref{splitt} that we get a decomposition of the regular representation $(V^n, A_\Sigma)$ into the invariant spaces 

\[W:= \ker(B_{\iota}) = \{X_{\iota \circ \sigma} = 0 \text{ } \forall \sigma \in \Sigma\} \quad, \]
and

\[W':=  \im(B_{\iota}) = \{X_{\sigma} = X_{\tau} \text{ if } \iota \circ \sigma = \iota \circ \tau \} \quad.\]
We will start by showing that 

\begin{equation}
W \cap \syn_{N,p} =  W \cap \syn_{N_P,[p]}
\end{equation}
It follows from remark \ref{NPpi} that 

\begin{equation}
 \syn_{N_P,[p]} =  \syn_{N_P,[p]} \cap  \syn_{\pi_P} =  \syn_{P} \cap \syn_{N,p}\, .
\end{equation}
Hence we see that

\begin{equation}
W \cap \syn_{N_P,[p]} \subset W \cap \syn_{N,p}  \, .
\end{equation}
Conversely, we note that
\begin{equation}\label{p1}
W \cap \syn_{N,p} = \{X_{\iota \circ \sigma} = 0 \text{ } \forall \sigma \in \Sigma\} \cap \{X_{\sigma} = X_{\tau} \text{ if } \sigma(p) = \tau(p)\} \, ,
\end{equation}
and 

\begin{equation}\label{p1dash}
W \cap  \syn_{N_P,[p]}= \{X_{\iota \circ \sigma} = 0 \text{ } \forall \sigma \in \Sigma\} \cap \{X_{\sigma} = X_{\tau} \text{ if } [\sigma(p)] = [\tau(p)]\} \, .
\end{equation}
To show that the space \ref{p1} is contained in \ref{p1dash}, we assume $X$ is an element in \ref{p1} and that $\sigma, \tau \in \Sigma$ are such that $ [\sigma(p)] = [\tau(p)]$. We then need to show that $X_{\sigma} = X_{\tau}$. There are two options. First of all it may be that $\sigma(p), \tau(p) \notin B$. Because the only partition class in $P$ possibly containing more than one node is $B$, we see that the equality  $[\sigma(p)] = [\tau(p)]$ implies $\sigma(p) = \tau(p)$. From this and the fact that $X \in \syn_{N,p}$ we conclude that indeed $X_{\sigma} = X_{\tau}$. \\
Next, we assume that $\sigma(p), \tau(p) \in B$. Because it will in general not hold that $\sigma(p) = \tau(p)$, we may not use  $X \in \syn_{N,p}$ to conclude that $X_{\sigma} = X_{\tau}$. Instead, we will show that $\sigma(p) \in B$ and $X \in W \cap \syn_{N,p}$ together imply that $X_{\sigma} = 0$. From this we then get $X_{\sigma} = 0 = X_{\tau}$ whenever $\sigma(p), \tau(p) \in B$, proving that indeed $X$ is an element of $W \cap \syn_{N_P,[p]}$.\\
Therefore, let $\sigma \in \Sigma$ be such that $b:=\sigma(p) \in B$. Because the map $\iota$ satisfies $\iota(C) = B$, there exists an element $c \in C$ such that $\iota(c) = b$. Applying $\iota$ to both sides and using that $\iota$ is idempotent we get $\iota(b) = \iota^2(c) = \iota(c) = b$. Hence we see that $\sigma(p) = (\iota \circ \sigma)(p)$. By the fact that $X \in \syn_{N,p}$ we see that $X_{\sigma} = X_{\iota \circ \sigma}$. However, by the fact that $X \in W$ it also follows that $X_{\iota \circ \sigma} = 0$. We conclude that indeed $X_{\sigma} = 0$.\\

\noindent Next we want to show that 

\begin{equation}
W'  \cap \syn_{\pi_P} = \syn_{0} \,,
\end{equation}
where $\syn_{0} := \{X_{\sigma} = X_{\tau}\text{ } \forall \sigma, \tau \in \Sigma\}$ is the fully synchronous space. First of all, because $W'$ and  $\syn_{\pi_P}$ are both synchrony spaces, it follows that 

\begin{equation}
W'  \cap \syn_{\pi_P} \supset \syn_{0} \quad.
\end{equation}
It remains to show that 
 
\begin{equation}
W'  \cap \syn_{\pi_P} \subset \syn_{0} \quad,
\end{equation}
hence that for all $X \in W'  \cap \syn_{\pi_P}$ we have $X_{\sigma} = X_{\tau}$ for all $\sigma, \tau \in \Sigma$. Therefore, let $\sigma$ and $\tau$ be given. From the identity $\iota \circ (\iota \circ \sigma) = \iota \circ (\sigma)$, we conclude that $X \in W'$ implies $X_{\sigma} = X_{\iota \circ \sigma}$. Likewise we find that  $X_{\tau} = X_{\iota \circ \tau}$. Next, by the fact that $\iota(C) = B$ we conclude that $[\iota] \in \Sigma_P$ is the function that sends every node in $N_P$ to the node $[B] \in N_P$. Since this holds equally well for $[\iota \circ \sigma]$ and $[\iota \circ \tau]$, we conclude that in fact $[\iota \circ \sigma] = [\iota \circ \tau]$. Finally it follows from $X \in  \syn_{\pi_P}$ that $ X_{\iota \circ \sigma} =  X_{\iota \circ \tau}$ and hence that $X_{\sigma} = X_{\tau}$, proving that indeed $X \in \syn_{0}$. This concludes the proof. 
\end{proof}

\begin{remk}
It follows from theorem \ref{submon} that the space $(V^m, A_{\Sigma_P}) \cong \syn_{\pi_P}$ admits a decomposition into invariant spaces

\begin{equation}
V^m = (W  \cap \syn_{\pi_P} )\oplus (W'  \cap \syn_{\pi_P}) = (W  \cap \syn_{\pi_P}) \oplus \syn_{0} \,.
\end{equation}
The space $ \syn_{0} $ can furthermore be decomposed into $\dim(V)$ indecomposable representations of $\Sigma_P$, on which this monoid acts trivially. Furthermore, if $U \subset  (V^m, A_{\Sigma_P}) $ is any indecomposable (complementable) representation, then $U$ is isomorphic to some indecomposable component of either $W  \cap \syn_{\pi_P}$ or $\syn_{0}$. Any such isomorphism can be expanded to an equivariant linear map from $(V^m, A_{\Sigma_P})$ to itself, for example by letting this map vanish on some complement of $U$. Since this map has the structure of a network map, it must send synchrony spaces to themselves. From this we conclude that if  $U \cap \syn_{0} = \{0\}$ then $U$ is isomorphic to some component of $W  \cap \syn_{\pi_P}$. Conversely, if  $U \cap \syn_{0} \not= \{0\}$ then it is isomorphic to a subspace of $\syn_{0}$, hence equal to a $1$-dimensional subspace of $\syn_{0}$. \hfill $\triangle$
\end{remk}
\noindent For the following corollary we note that any network vector field for the network $N_P$ can be realized as the restriction to $\syn_{N_P,[p]}$ of a network vector field on $(V^m, A_{\Sigma_P})$ and hence of a network vector field on $(V^n, A_{\Sigma})$. Restricting such a lifted vector field on $(V^n, A_{\Sigma})$ to $\syn_{N, p}$, we furthermore see that a network vector field for $N_P$ can always be seen as the restriction of a network vector field for the original network $N$. Similarly, any smooth family of network vector fields for $N_P$ can be lifted to a smooth family of network vector fields for $N$, on $(V^m, A_{\Sigma_P})$ and on $(V^n, A_{\Sigma})$. 

\begin{cor}\label{main2}
Let $N$ be a homogeneous coupled cell network with projection block $B$ and corresponding quotient network $N_P$.  Let $\gamma_f:  V^{\# C_P} \times \Omega \rightarrow V^{\# C_P}$ be a family of smooth network vector fields for the network $N_P$, indexed by $\Omega \subset \R^k$ with $0 \in \Omega$. Suppose furthermore that $\gamma_f$ satisfies $\gamma_f(0,0) = 0$ and suppose the center subspace $W_c$ of the linearization $D_x\gamma_f(0,0)$ satisfies $W_c \cap \syn_0 = \{0\}$. We denote by $\Lambda_f$ the set of smooth network vector fields for $N$, $\tilde{\gamma}_g: V^{\#C} \times \Omega \rightarrow V^{\#C}$, such that $\gamma_f = \tilde{\gamma}_g|_{\syn_{P}\times \Omega}$. Then, there exists an open dense set $U$ of $\Lambda_f$ such that for all $\tilde{\gamma}_g \in U$ it holds that the locally defined center manifold of $\gamma_f$ around the origin is a local center manifold for $\tilde{\gamma}_g$ around the origin. In particular, any (local) bifurcation occurring in $\gamma_f$ is then a bifurcation occurring in $\tilde{\gamma}_g$, without any additional bounded solutions appearing in this latter system.  
\end{cor}

\begin{proof}
We let $\Lambda'_f$ denote the set of smooth fundamental vector fields \\
 $\Gamma_h: (V^n, A_{\Sigma}) \times \Omega \rightarrow (V^n, A_{\Sigma})$ such that $\gamma_f = \Gamma_h|_{\syn_{N_P,[p]}\times \Omega}$. Note that $\Gamma_h$ is in $\Lambda'_f$ if and only if $\Gamma_h|_{\syn_{N, p}\times \Omega}$ is in $\Lambda_f$. We pick an element $\tilde{\gamma}_g$ in $\Lambda_f$ and a corresponding element $\Gamma_h$ in $\Lambda'_f$ with $\Gamma_h|_{\syn_{N, p}\times \Omega} = \tilde{\gamma}_g$. Next, we decompose $(V^n, A_{\Sigma})$ as in theorem \ref{main}: 

\begin{equation}
 V^n =  W \oplus W'\,,
\end{equation}
with 

\begin{equation}
W \cap \syn_{N,p} =  W \cap \syn_{N_P,[p]}
\end{equation}
and 
\begin{equation}\label{inte5}
W'  \cap \syn_{\pi_P} = \syn_{0} \,.
\end{equation}
Let $W'_c$ be the center subspace of $D_x\Gamma_h(0,0)$. We write

\begin{equation}
W'_c = \bigoplus_{i=1}^l  U_i
\end{equation}
as the decomposition into indecomposable representations. If the component $U_i$ has a trivial intersection with $\syn_{\pi_P}$ then we may add an equivariant linear map that vanishes on the complement of $U_i$ to remove it from $W'_c$ without influencing $\gamma_f$. Hence we see that for an open dense set of $\Lambda'_f$, and hence of $\Lambda_f$, no such components are present. If $W'_c$ now has any components isomorphic to a component of $W'$, then by equation \eqref{inte5} its (non-trivial) intersection with $\syn_{\pi_P}$ is in $\syn_0$, contradicting the assumption that $W_c \cap \syn_0 = W'_c \cap \syn_0 = \{0\}$. We see that $W'_c$ is therefore isomorphic to a subspace of $W$. In \cite{new} it is shown that the center manifold of $\Gamma_h$ can be seen as the image of an equivariant map from $W'_c \times \R^k$ to $V^n \times \R^k$. Because this map furthermore preserves synchrony spaces, we find the center manifolds of $\gamma_f$ and $\tilde{\gamma}_g$ by restricting this map to $(\syn_{N_P,[p]} \cap W'_c) \times \R^k$ and $(\syn_{N,p} \cap W'_c) \times \R^k$, respectively. However, since $W'_c$ is isomorphic to a subspace of $W$ and because equivariant isomorphisms preserve synchrony spaces, we have that 

\begin{equation}
W'_c \cap \syn_{N,p} =  W'_c \cap \syn_{N_P,[p]}\,.
\end{equation}
This proves that the center manifolds, and hence the bifurcations agree.
\end{proof}

\section{Example: Ring feed-forward networks}
In this section we will apply the machinery we have developed so far to a generalization of the feed forward network. It will turn out that the reduced vector fields of this network can be completely understood by that of two of its quotient networks.

\begin{defi}
The $(n,k)$-\textit{ring feed-forward network} $R_{n,k}$ is the homogeneous coupled cell network with nodes $C := \{c_i\}_{i=0}^{n+k-1}$ and with monoid $\Sigma$ generated by a single element $\sigma$. This element is given on the nodes by $\sigma(c_i) = c_{i+1}$ for $i <n+k-1$ and $\sigma(c_{n+k-1}) = c_k$. We will collectively refer to the $(n,k)$-ring feed-forward networks as simply the \textit{ring feed-forward networks}.
\end{defi}

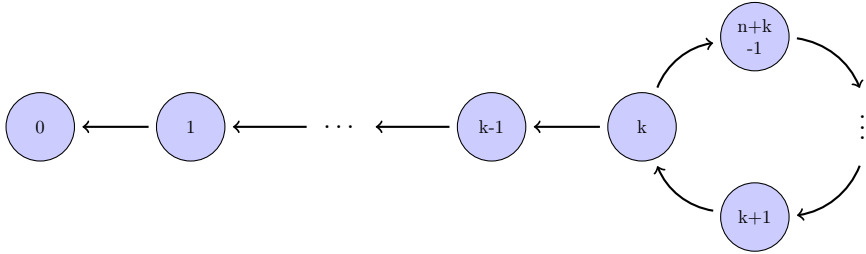
\begin{figure}[h]
\centering
\begin{tikzpicture}
\node[circle, draw, fill=blue!20, scale=0.7, minimum size=1.3cm] at (0,0) (a) {0} ; 
\node[circle, draw, fill=blue!20, scale=0.7, minimum size=1.3cm] at (2,0) (b) {1};
\node[] at (4,0) (c) {\dots} ;
\node[circle, draw, fill=blue!20, scale=0.7, minimum size=1.3cm] at (6,0) (d) {k-1}  ;
\node[circle, draw, fill=blue!20, scale=0.7, minimum size=1.3cm] at (8,0) (e) {k}  ;
\node[circle, draw, fill=blue!20, scale=0.7, minimum size=1.3cm] at (9.5,-1.2) (f) {k+1}  ;
\node[align=left] at (11,-0.3) (g1) {} ;
\node[align=left] at (11,0.3) (g2) {} ;
\node[align=left] at (10.93,0.1) (g) {\vdots}  ;
\node[circle, draw, fill=blue!20,align=center,  scale=0.7, minimum size=1.3cm] at (9.5,1.2) (h) {n+k \\ -1}  ;
\draw[<-, thick, shorten >= 0.1 cm, shorten <= 0.1 cm] (a) edge (b) (b) edge (c) (c) edge (d) (d) edge  (e) (e) edge[bend right] (f) (f) edge[bend right] (g1) (g2) edge[bend right] (h) (h) edge[bend right] (e);
\end{tikzpicture}\caption{The $(n,k)$-ring feed-forward network $R_{n,k}$.}\label{fig:pic3}
\end{figure}

\noindent We claim that the set of nodes $B:=\{c_k, \dots c_{n+k-1}\} \subset C$ is a projection block in the network $R_{n,k}$. Indeed, it is clearly a block as we have $\sigma(B) = B$, and therefore $\sigma^i(B) = B$ for all $i \geq 0$. To see that $B$ is a projection block, recall that the monoid $\Sigma$ is generated by the single element $\sigma$. By theorem \ref{pict}, $B$ is then a projection block if and only if $\sigma$ restricts to a bijection on $B$ and any element outside of $B$ is sent to $B$ by some power of $\sigma$. By the definition of the network $R_{n,k}$ these two conditions are indeed satisfied. \\
By identifying the block $B$ with a point, we get a network with $k+1$ cells $C_P:=\{[c_0], \dots [c_{k-1}], [B] \}$ and with a monoid $\Sigma_P$ generated by the element $[\sigma]$. This element is given on $C_P$ by $[\sigma]([c_i]) = [c_{i+1}]$ for $i < k-1$ and $[\sigma]([c_{k-1}]) = [\sigma]([B]) = B$. In particular, we see that this network is equal to $R_{1,k}$. This latter network is known in the literature as a feed-forward network, and the bifurcations of its admissible vector fields are quite well understood, see \cite{feedf}. In particular, setting $V = \R$ and $\Omega \subset \R$ it is known that the generic steady state bifurcations of $R_{1,k}$ from a fully synchronous point are either a fully synchronous saddle node bifurcation or a synchrony breaking bifurcation. In this latter bifurcation there are, in addition to a fully synchronous branch, $k$ branches scaling as $|\la|^{l_1}$ to $|\la|^{l_k}$, where we have $l_i := \frac{1}{2^{i-1}}$. \\
Furthermore, if we set $V = \C$ and $\Omega \subset \R$ then the feed-forward network admits a synchrony breaking Hopf bifurcation supporting $k$ branches of periodic orbits with amplitudes scaling as $|\la|^{p_1}$ to $|\la|^{p_k}$, where we have set $p_i := \frac{1}{2(3^{i-1})}$.\\
Of the bifurcations just described, the synchrony breaking ones correspond to a center subspace with trivial intersection with the fully synchronous space $\syn_0$. Hence, by the previous chapter we conclude that these synchrony breaking bifurcations occur in the network $R_{n,k}$ as well. To determine the bifurcations corresponding to other indecomposable representations, let us have a more detailed look at the network $R_{n,k}$.

\begin{thr}\label{sameas}
Any ring feed-forward network is isomorphic to its own fundamental network.
\end{thr}

\begin{proof}
We begin by noting that $c_0$ is a fully dependent node for any network $R_{n,k}$. Hence, every ring feed-forward network is a quotient network of its fundamental network. It remains to show that the number of nodes of the fundamental network is equal to that of the corresponding ring feed-forward network. In other words, since the nodes of the fundamental network are the elements of the monoid $\Sigma$, we need to show that $\#\Sigma = \#C$. However, in $R_{n,k}$ we know that $\sigma^{n+k}(c_0) = \sigma^{k}(c_0) = c_k$, from which it follows that

\begin{equation}
\begin{split}
 \sigma^{n+k}(c_i) &=  \sigma^{n+k}\sigma^i(c_0) =  \sigma^i\sigma^{n+k}(c_0) \\
\sigma^i\sigma^{k}(c_0) &= \sigma^k\sigma^{i}(c_0) = \sigma^{k}(c_i) \quad.
\end{split}
\end{equation}
Hence we see that $\sigma^{n+k} = \sigma^{k}$ as functions, from which it follows that $\#\Sigma \leq n+k = \#C$. The fact that $\#\Sigma \geq \#C$ follows from the fact that $R_{n,k}$ can be realized as a quotient network of its fundamental network or from the fact that the functions $\sigma^i$ send $c_0$ to different nodes $c_i$ for $i = 0, \dots k+n-1$. This proves the theorem.
\end{proof}
\noindent In light of theorem \ref{sameas} we may think of $R_{n,k}$ as its own fundamental network. In doing so, we will write $\Sigma = C: = \{\sigma^0=\id, \sigma^1, \dots \sigma^{n+k-1}\}$ where an element of $\Sigma$ acts on an element of $C$ by composition. Let us furthermore write $(V^{n+k}, A_\Sigma)$ for the regular representation of $R_{n,k}$. The equivariant vector fields on $(V^{n+k}, A_\Sigma)$ are then exactly the admissible vector fields of $R_{n,k}$. Furthermore, the regular representation space $(V^{k+1}, A_{\Sigma_P})$ for the network $R_{1,k}$ can be realized as an invariant robust synchrony space $(V^{k+1}, A_{\Sigma_P}) = \syn_{\pi_P} \subset (V^{n+k}, A_\Sigma)$. More specifically, $\syn_{\pi_P}$ is given by 

\begin{equation}
\syn_{\pi_P} = \{X_{\sigma^i} = X_{\sigma^j} \forall \text{ } i,j \in \{k, \dots k+n-1 \} \} \quad.
\end{equation}
Because $B$ is a projection block, it follows from lemma \ref{idemp} that there exists an idempotent element $\sigma^{T} \in \Sigma$ such that $\sigma^{T}(C) = B$. By lemma \ref{idemproj} this element gives rise to a projection $A_{\sigma^T}$ on $(V^{n+k}, A_\Sigma)$  given by $(A_{\sigma^T}X)_{\sigma^i} := X_{\sigma^{i+T}}$, where we use the convention of writing $X_{\sigma^i} = X_{\sigma^j}$ if $i,j \geq k$ and $n|(i-j)$. Furthermore, because $\sigma^T$ satisfies $\sigma^{T}(C) = B$, it necessarily follows that $T \geq k$. From this we see that 

\begin{equation}
\begin{split}
&\im(A_{\sigma^T}) = \{X_{\sigma^i} = X_{\sigma^j} \text{ if } n|(i-j) \} \quad, \\
&\ker(A_{\sigma^T}) = \{X_{\sigma^i} = 0 \text{ } \forall \text{ } i \in \{k, \dots k+n-1 \}  \}\quad.
\end{split}
\end{equation}
Note that $\ker(A_{\sigma^T})$ is contained in $\syn_{\pi_P}$, which accounts for the bifurcations occurring in both $R_{1,k}$ and $R_{n,k}$, as we have found by theorem $\ref{main}$ and corollary \ref{main2}. Furthermore, if $V = \R$ then it can be shown that $\ker(A_{\sigma^T})$ is an indecomposable representation of $\Sigma_P$, and hence of $\Sigma$. Note furthermore that $\Sigma$ acts on $\ker(A_{\sigma^T})$ by nilpotent maps, a fact that we will use later on. \\
To summarize so far, we know that $(V^{n+k}, A_\Sigma)$ decomposes into the invariant spaces $\im(A_{\sigma^T})$ and $\ker(A_{\sigma^T})$. The bifurcations corresponding to $\ker(A_{\sigma^T})$ are now related to the synchrony breaking bifurcations in $R_{1,k}$, and this was ultimately done by noting that $R_{1,k}$ is a quotient network of $R_{n,k}$. We will now explain the bifurcations in $R_{n,k}$ corresponding to $\im(A_{\sigma^T})$ by considering yet another, well understood quotient network of $R_{n,k}$.

\begin{thr}\label{znz}
Let $\Sigma$ denote the monoid of $R_{n,k}$. The map
\begin{equation}
\begin{split}
\pi_n: \text{}&\Sigma \rightarrow \Z /n\Z \\
&\sigma^i \mapsto [i] 
\end{split}
\end{equation}
realizes $\Z /n\Z$ as a quotient monoid of $\Sigma$. As a result, the regular representation of $\Z /n\Z$ is realized in $(V^{n+k}, A_\Sigma)$ as the space \\ 
$\im(A_{\sigma^T}) = \{X_{\sigma^i} = X_{\sigma^j} \text{ if } n|(i-j) \}$.
\end{thr}

\begin{proof}
First of all we see that the map $\pi_n$ is well defined, as $\sigma^i = \sigma^j$ implies $n|(i-j)$. Next, it is clear that the map is a surjective morphism of monoids. I.e we have that $\pi_n(\sigma_{i} \circ \sigma_{j}) = [i+j] = [i]+[j]= \pi_n(\sigma^i) + \pi_n(\sigma^j)$. Finally, by theorem \ref{inclu} we see that the regular representation of $\Z/n\Z$ is isomorphic to the representation of $\Sigma$ in $(V^{n+k}, A_{\Sigma})$ restricted to the invariant subspace $\syn_{\pi_n}:= \{X_{\sigma^i} = X_{\sigma^j} \text{ if } \pi_n(\sigma^i) = \pi_n(\sigma^j)\}$. Because $\pi_n(\sigma^i) = \pi_n(\sigma^j)$ if and only if $n|(i-j)$, we conclude that $\syn_{\pi_n} = \im(A_{\sigma^T})$. This proves the theorem.
\end{proof}
It follows from theorems \ref{znz} and \ref{submon} that the indecomposable complementable sub-representations of $(V^{n+k}, A_{\Sigma})$ contained in $\im(A_{\sigma^T})$ are exactly the irreducible sub-representations of the regular representation of $\Z/n\Z$. These are well understood. Furthermore, the irreducible representations of $\Z/n\Z$ are mutually non isomorphic and are as representations of $\Sigma$ non isomorphic to $\ker(A_{\sigma^T})$, as $\sigma$ acts as a nilpotent map in the latter representation. This proves that the generic steady state bifurcations in a fully synchronous point of $R_{n,k}$, given that $V= \R$ and $\Omega \subset \R$, are exactly given by the synchrony breaking bifurcation of $R_{1,k}$ and the generic steady state bifurcations of $\Z/n\Z$.

\end{document}